\documentclass[12pt,leqno]{amsart}
\usepackage{times}
\usepackage{amsmath,amssymb,amscd,amsthm}
\usepackage{latexsym}
\usepackage{epsfig}

\newtheorem{theorem}{Theorem}
\newtheorem{lemma}{Lemma}
\newtheorem{proposition}{Proposition}

\newtheorem{definition}{Definition}
\newtheorem{corollary}{Corollary}

\newtheorem{claim}{Claim}

\newcommand{\f}[2]{\frac{#1}{#2}}

\newcommand{\de}{\delta}

\newcommand{\la}{\lambda}

\newcommand{\La}{\Lambda}

\newcommand{\Si}{\Sigma}

\newcommand{\xsb}{X^{s,b}}
\newcommand{\ysb}{Y^{s,b}}

\newcommand{\rone}{\mathbf R^1}

\newcommand{\zz}{\mathbf Z}

\newcommand{\cm}{\mathcal M}

\newcommand{\cc}{\mathcal C}

\newcommand{\cx}{\mathcal X}

\newcommand{\p}{\partial}

\newcommand{\beq}{\begin{equation}}
\newcommand{\eeq}{\end{equation}}
\newcommand{\beqna}{\begin{eqnarray*}}
\newcommand{\eeqna}{\end{eqnarray*}}
\newcommand{\beqn}{\begin{equation*}}
\newcommand{\eeqn}{\end{equation*}}
\newcommand{\bp}{\begin{proof}}
\newcommand{\ep}{\end{proof}}
\newcommand{\bprop}{\begin{proposition}}
\newcommand{\eprop}{\end{proposition}}
\newcommand{\bt}{\begin{theorem}}
\newcommand{\et}{\end{theorem}}
\newcommand{\bex}{\begin{Example}}
\newcommand{\eex}{\end{Example}}
\newcommand{\bc}{\begin{corollary}}
\newcommand{\ec}{\end{corollary}}
\newcommand{\bcl}{\begin{claim}}
\newcommand{\ecl}{\end{claim}}
\newcommand{\bl}{\begin{lemma}}
\newcommand{\el}{\end{lemma}}

\begin{document}

\title
[l.w.p. for periodic  mKdV]
{Local well-posedness for  the periodic   mKdV in  $H^{1/4+}$ }

\author{ Atanas Stefanov}

\address{Atanas Stefanov, 
405, Snow Hall, 1460 Jayhawk Blvd. , 
Department of Mathematics,
University of Kansas,
Lawrence, KS~66045, USA}
\email{stefanov@ku.edu}
\date{\today}

\thanks{The author gratefully acknowledges partial support by  NSF-DMS \# 1313107. }

\subjclass[2000]{35Q53 (35B10, 35B30)}

\keywords{Modified KdV, periodic boundary conditions, low regularity well-posedness.}

\begin{abstract}
We study  the mKdV equation with periodic boundary conditions. We establish low regularity well -posedness in $H^{\f{1}{4}+}(T)$. The proof involves a non-linear, solution dependent gauge transformation, similar to the one considered in \cite{NTT}. 
\end{abstract}

\maketitle

\section{Introduction} 
{\bf The main result of this paper is invalidated, due to the failure of  the estimate $(16)$ below.}
Consider the real-valued modified Korteweg-de Vries equation with periodic boundary condition
\begin{equation}\label{mkdv}
\left| \begin{array}{l} u_t + u_{xxx} +u^2 \p_x u=0,\\
	u(0)= f \in H^{s} (\mathbf{T}).\end{array} \right.
\end{equation}
Note that if $f$ is real-valued, then 
$$
f(x)=\sum_{k=-\infty}^\infty \hat{f}(k) e^{-2\pi i k x}, \ \ \overline{\hat{f}(k)}=\hat{f}(-k). 
$$
Even though there were quite a few results dealing with the well-posedness of this model with standard energy methods, it was Bourgain,   who has initiated in \cite{Bou1}, the study of the well-posedness of such models at low regularity. The main new technical idea was the introduction of adapted to the evolution function spaces (coined  $X^{s,b}$ spaces), which are more sensitive than the standard energy spaces for the problems at consideration. We should mention that in the case of the problem on $\rone$, better results are achieved by using the local smoothing estimates associated with the Airy equation, as shown in \cite{CPV4}. 

The problem for obtaining local well-posedness in spaces with less and less Sobolev regularity has received lots of attention by many auhors in the last twenty years.  Since Bourgain has showed his basic trilinear estimate (which coupled with his method  gives the local well-posedness in $H^{1/2}(\mathbf{T})$), it was shown by Kenig-Ponce-Vega, \cite{CPV1}  that this estimate actually  fails in $H^s (\mathbf{T}), s<1/2$. In fact, not only this estimate fails, but the solution map was shown to be not uniformly continuous when $f\in H^s(\mathbf{T}), s<1/2$, \cite{CCT}. 

However, this does not necessarily mean that the local well-posedness fails.  Takaoka-Tsutsumi, \cite{TT} have considered the problem in $H^s, s>3/8$ and they have shown the local well-posedness, by using an iteration argument in $X^{s,b}$ type spaces, which depends on the initial data. This results were further extended in the work of Nakanishi-Takaoka-Tsutsumi, \cite{NTT}, where the authors have been able to push the l.w.p.  results to  $H^{1/3+}(\mathbf{T})$. Note that the authors have been able to provide existence results in $H^{1/4+}$, under some additional restrictions on the growth of the Fourier coefficients of the data. The main goal of this paper is to consider data in $H^{\f{1}{4}+}(\mathbf{T})$ and to show local well-posedness.

We start with some standard reductions. For nice solutions $u$ of \eqref{mkdv}, we have conservation of $L^2$~norm.  By changing the spatial variable $x$ to $x+ct$ where $c= \frac{1}{2\pi} \|u_0\|_{L^2}^2$, we have 
\begin{equation}
\label{a:10}
\left|\begin{array}{l}
\p_t u + \p_x^3 u  + (u^2 - \frac{1}{2\pi} \int_{\mathbf{T}} u^2 (t,x)\, dx) \partial_x u =0\\
u(0) = f.\end{array}\right.
\end{equation}
This is the equation that we consider from now on. 
On the Fourier side, the equation is\footnote{For more details about this derivation, the reader may consult \cite{NTT}, p. 1639. } 
$$
\p_t \widehat{u}(t,k) - i k^3 \widehat{u}(t,k) = -i\frac{k}{3}  \sum_{\tiny\begin{array}{c} k_1 + k_2 + k_3 = k,\quad k_j, k\neq 0\\ (k_1 + k_2)(k_2 + k_3)(k_3+ k_1) \neq 0\end{array}} \widehat{u}(k_1) \widehat{u}(k_2) \widehat{u}(k_3) + ik|\widehat{u}(k)|^2 \widehat{u}(k).
$$ 
The first term is called non-resonant, while the other term is referred to as resonant. The non-resonant trilinear   term $\mathcal{NR}$     is introduced to be 
$$
\mathcal{NR}(v_1, v_2, v_3)(k):=
 -i\frac{k}{3}  \sum_{\tiny\begin{array}{c} k_1 + k_2 + k_3 = k,\quad k_j, k\neq 0\\ (k_1 + k_2)(k_2 + k_3)(k_3+ k_1) \neq 0\end{array}} \widehat{v_1}(k_1) \widehat{v_2}(k_2) \widehat{v_3}(k_3)
$$ 
We will sometimes denote $\mathcal{NR}(h):=\mathcal{NR}(h,h,h)$.

\subsection{Change of variables} 
We start with a general discussion  about the change of variables that is required. Basically, one needs to hide the resonant term $ ik|\widehat{u}(k)|^2 \widehat{u}(k)$. To that end,  introduce the change of variables, 
$$
\hat{u}(t,k):=\hat{v}(t,k) + \hat{f}(k)e^{i(t k^3+k \int_0^t |\hat{u}(s,k)|^2 ds)}.
$$
Denote for convenience  $P(t,k):=t k^3+k \int_0^t |\hat{u}(s,k)|^2 ds$.  This would transform the equation into a new one for $v$, in the form 
\begin{equation}
\label{v}
\begin{array}{rl}
\p_t \hat{v}(k)-i(k^3+k|\hat{f}(k)|^2)\hat{v}(k)  &=   i k |\hat{v}(t,k)|^2 \hat{v}(t,k)+ \\
&+ 2i k \Re(\hat{f}(k) e^{i P(t,k)}\overline{\hat{v}(t,k)}) \hat{v}(k)+ \\
 &+ \mathcal{NR}(\otimes_{j=1}^3  \hat{f}(k_j) e^{i  P(t,k_j)} +\hat{v}(k_j))\\
v(0,k) & =  0 
\end{array}
\end{equation}
The disadvantage of this equation for $v$ is that the old variable $u$ is still present inside at the phase function $P$. Nevertheless, for uniqueness purposes, it is good to consider  exactly \eqref{v}.  

For existence results however,  we  seek to introduce a new variable $z$, so that the phase variable (denoted $Q$ below)  is dependent only upon the new variable $z$ and which does not contain a reference to the old one $u$.   We need the following  
\begin{lemma}
\label{lo:1}
Let $f\in H^{s_0}(T)$, $s_0>0$.  Let $\{\hat{z}(t,k)\}_k$ are given continuous functions, defined on an interval $[0,T]$. Assuming that there exists $C$, so that 
\begin{equation}
\label{k:90}
\sup_{0<t<T} \sup_{k} <k>^{1-s_0}  |\hat{z}(t,k)|\leq C .
\end{equation}
then for  the infinite system of (non-linear) ODE's 
\begin{equation}
\label{k:100}
Q'(z;t,k)=k^3+k| \hat{f}(k)e^{i Q(z;t,k)}+\hat{z}(t,k)|^2,  Q(z;0,k)=0, \ \ k\in \zz
\end{equation}
there exists a time interval $[0,T_0]$ ,  $ T_0\geq 
\min(T, \f{1}{100 C_0\|f\|_{H^{s_0}}})$, so that it  has unique solution $\{Q(z;k,t)\}_{k\in zz} :[0,T_0]\to \rone$.    
In particular, the condition \eqref{k:90} is satisfied if $z=\sum_k \hat{z}(t,k)e^{i k x}\in L^\infty_t H^{1-s_0}$. 
\end{lemma}
\noindent {\bf Remark:}  For the most part, we will suppress the dependence of $Q$ on $z$ in our notations. 
\begin{proof}
The existence argument is easy and it can be justified, based on the theory of nonlinear ODE with Lipschitz right hand sides.  The non-trivial part of the statement is the common interval of existence, which is independent of $k$. 

  To that end, rewrite the system of ODE's as equivalent system of integral equations 
\begin{equation}
\label{k:95}
Q(t,k)= t (k^3+k |\hat{f}(k)|^2) + k\int_0^t (2\Re(\hat{f}(k) e^{ i Q(s,k)} \overline{\hat{z}(s,k)}) + |\hat{z}(s,k)|^2) ds 
\end{equation} 
In order to check that the fixed point argument produces a solution in an interval $[0, T_0]$, we need to check  the contractivity of $Q\to  \Sigma(Q):=k \int_0^t (2\Re(\hat{f}(k) e^{ i Q} \overline{\hat{z}(s,k)}) ds$. Indeed, 
\begin{eqnarray*}
 \sup_{0<t<T_0} |\Si(Q_1)(t)-\Si(Q_2)(t)| &\leq & 10 T_0 |k| |\hat{f}(k)| \sup_{0<s<T_0} 
|Q_1(s)-Q_2(s)| \sup_{0<s<T_0}|\hat{z}(s,k)|< \\
& \leq & 10 \|f\|_{H^{s_0}} C_0 T_0 \sup_{0<s<T_0} \|Q_1(s)-Q_2(s)\|,
\end{eqnarray*}
since 
$$
|k| |\hat{f}(k)|  \sup_{0<\tau<T_0}  |\hat{z}(\tau,k)|\leq  C \|f\|_{H^{s_0}} \sup_{k,\tau} <k>^{1-s_0} |\hat{z}(\tau,k)|\leq C.
$$

It follows that $\Si$ is a contraction, whenever $T_0<1/(20 C_0 \|f\|_{H^{s_0}}), T_0<T$ and the lemma is proved.
\end{proof}
We now continue with the precise definition of the transformation. In the new variable $z:[0,T]\to \cc$, let  $Q=Q_z$ as in Lemma \ref{lo:1}. That is, let $Q$ be the solution of \eqref{k:100}. Clearly, 
$z$  needs to be in $H^{1-s_0}$, which will be established a-posteriori. 
Set  
$$
\hat{u}(t,k):=\hat{z}(t,k) + \hat{f}(k) e^{i  Q(t,k)}.
$$
Note  $\hat{z}(0,k)=0$, since $\hat{u}(0,k)=\hat{f}(k), Q(0,k)=0$. 
In terms of $z$, the equation {\it equivalent to the original equation} \eqref{a:10}   becomes 
\begin{equation}
\label{z}
\begin{array}{rl}
\p_t \hat{z}(k)-i(k^3+k|\hat{f}(k)|^2)\hat{z}(k)  &=   i k |\hat{z}(t,k)|^2 \hat{z}(t,k)+ \\
&+ 2i k \Re(\hat{f}(k) e^{i Q(t,k)}\overline{\hat{z}(t,k)}) \hat{z}(k)+ \\
 &+ \mathcal{NR}(\otimes_{j=1}^3  \hat{f}(k_j) e^{i  Q(t,k_j)} +\hat{z}(k_j))\\
z(0,k)  =& 0 
\end{array} 
\end{equation}
We are now ready to give the definition of local existence  that we will be working with.  
\begin{definition}
\label{defi:1}
Let $1>s_0>0$ and $f\in H^{s_0}(T)$. We say that $u$ is a solution to the  mKdV equation, with initial data $f$, if there exists $T>0$ and $z(t,x)\in L^\infty(0,T)  H^{1-s_0}_x$ so that the pair  $z$ and the unique $Q=Q(z;t):[0,T_0]\to \rone$ produced by Lemma \ref{lo:1} satisfy  the preceding equation in strong sense. More precisely, 
\begin{eqnarray*}
  \hat{z}(t,k) &=& \int_0^t e^{i (t-s) (k^3+k|\hat{f}(k)|^2)} [i k |\hat{z}(s,k)|^2 \hat{z}(s,k)+  2i k \Re(\hat{f}(k) e^{i Q(s,k)}\overline{\hat{z}(t,k)}) \hat{z}(k)] ds + \\
& + & \int_0^t e^{i (t-s) (k^3+k|\hat{f}(k)|^2)} [\mathcal{NR}(\otimes_{j=1}^3  \hat{f}(k_j) e^{i  Q(t,k_j)} +\hat{z}(k_j))]ds.
\end{eqnarray*}
\end{definition}

\subsection{Function spaces} 
\label{sec:fs}
Since we study a local well-posedness question, we introduce  function spaces, in which the solutions   will live. Naturally, these will be versions of the ubiquitous Bourgain spaces, initially defined for the pure KdV  evolution for  functions on the torus $z: \rone\times \mathbf{T}\to \cc, 
z(t,x)=\sum_k z_k(t) e^{ i k x}$  
$$
\|z\|_{\xsb}^2 =  \sum_{k} \int_{\rone}   <\tau-k^3>^{2b} <k>^{2s} 
|\hat{z}(\tau,k)|^2   d\tau. 
$$
In addition,  we introduce the   modified Bourgain space $\ysb$    as follows 
\begin{eqnarray*}
\|z\|_{\ysb}^2 &=&  \sum_{k} \int_{\rone}   <\tau-k^3-k|\hat{f}(k)|^2>^{2b} <k>^{2s} |\hat{z}(\tau,k)|^2   d\tau. 
\end{eqnarray*}
It will also be convenient to use the local version of these spaces, namely for any $T>0$,   define (for any $\La=\xsb, \ysb $) 
$$
\|v\|_{\La_T}=\inf\{\|u\|_\La, u\in \La, u=v \ \textup{on}\ \ (-T,T) \}
$$
For the remainder of this paper we will tacitly assume that $T<1$. 
\subsection{Main result} 
 The following is the main result of this work. 
\begin{theorem}
\label{theo:1} 
Let $s_0>\f{1}{4}$ and $0<\de<<s_0-\f{1}{4}$, $f\in H^{s_0}(\mathbf{T})$. Then, there exists  a solution in the sense of Definition \ref{defi:1}.    In addition, we have the following smoothing effects: 
\begin{eqnarray}
\nonumber
& & \sum_k \left[\hat{u}(t,k)-\hat{f}(k) e^{ i (t k^3+k\int_0^t |\hat{u}(s,k)|^2 ds)}\right]e^{i k x} \in L^\infty_t H^{3s_0-}, \\
\label{a:1050}
& & \sum_k |k| ||\hat{u}(t,k)|^2-|\hat{f}(k)|^2|<\infty. 
\end{eqnarray}
 Assuming that  $u\in L^2(\mathbf{T})$  obeys 
\begin{equation}
\label{cond:10}
\sup_k |k| |\hat{u}(t,k)|^2-|\hat{f}(k)|^2|<\infty, 
\end{equation} 
 the equation \eqref{v} has an unique solution $v$, which is in $Y^{s_0,b}\cap L^\infty H^{3s_0-}$. 

The uniqueness holds in the following sense -  let  $v_1, v_2$ be  the two solutions of \eqref{v}, corresponding to $u_1, u_2\in  L^\infty_T H^{s_0}(\mathbf{T}) $ and 
satisfying \eqref{cond:10}, with  $u_j(0)=f$, then there exists $\tilde{T}>0$, so that 
$v_1|_{[0,\tilde{T}]}=v_2|_{[0,\tilde{T}]}$. 
\end{theorem}
{\bf Remark:} We can upgrade \eqref{a:1050} to 
\begin{equation}
\label{a:1060}
\sum_k |k|^{\min(4s_0, 1+s_0)}   |  |\hat{u}(t,k)|^2-|\hat{f}(k)|^2|<\infty.
\end{equation}
  One should compare the smoothing condition \eqref{a:1060} to the smoothing condition \eqref{cond:10}, which was proved  in \cite{NTT}, under the assumption $s_0>1/3$. 

Let us outline the plan for the paper. In Section \ref{sec:2}, we give some preliminary estimates, including an adaptation of the trilinear Bourgain estimate for the non-resonant terms. In Section \ref{sec:3}, we give the main estimates in this work, which quantify the smoothing of the non-resonant terms as well as the contribution of the resonant terms. In Section \ref{sec:4}, we put together the estimates from Section \ref{sec:3}, to justify an iteration argument, which provides the existence of the solution $z$ of \eqref{z} (and hence of $u$). Then, we show that the equation \eqref{v} has unique solution, for fixed $u$. This is however not enough for uniqueness, but shows that the correspondence $u\to v$ is well and uniquely defined. Finally, for uniqueness, we show that if two solutions $u_1, u_2$, with common initial data $f$ produce $v_1, v_2$, then $v_1=v_2$ in some eventually smaller time interval and hence $u_1=u_2$. 
 
\section{Preliminary  estimates}
\label{sec:2}

We have the following linear estimate. 
\begin{lemma}
\label{le:19}
Let  $z$ solves the following equation 
$$
\p_t z_k(t)- i (k^3 +  k|\hat{f}(k)|^2) z_k(t)=F_k(t). 
$$
in the sense that 
$$
z_k(t)=e^{i t (k^3+ k  |\hat{f}(k)|^2)} z_k(0)+ 
\int_0^t e^{i (t-s) (k^3+ k  |\hat{f}(k)|^2)} F_k(s) ds.
$$
Then for every $\de>0$, 
\begin{eqnarray*}
& & \| z\|_{\ysb_T}\leq C_{\de} T^\de (\|z(0,x)\|_{H^s({\mathbf T})}+ 
\|F\|_{Y_T^{s,b-1+\de}}).  
\end{eqnarray*}
\end{lemma} 
 
We now state a straightforward  extension of a well-known estimate by Bourgain, which will be crucial for our approach  in the sequel. More precisely, it was proved\footnote{although not explicitly stated,    see the remarks (b) after Proposition 8.37}  that 
\begin{equation}
\label{a:200}
\|\mathcal{NR}(u_1, u_2, u_3)\|_{X^{s,-1/2}}\leq C \|u_1\|_{X^{s,1/2}}\|u_2\|_{X^{s,1/2}}\|u_3\|_{X^{s,1/2}}
\end{equation}
 whenever $s>1/4$. Similar estimate, with $X^{s,b}$ replaced by $\ysb$, was established by \cite{NTT}, see  Lemma 2.2, p. 3017.    We need a  variant of \eqref{a:210}, namely 
\begin{lemma}
\label{le:bou}
Let  $s>1/4, b>1/2$ and $0<\de<<s-1/4$. Then, there exists a constant $C=C_{\de}$, so that 
\begin{equation}
\label{a:210}
\|\mathcal{NR}(u_1, u_2, u_3)\|_{Y^{s,b-1+\de}}\leq C_{b,\de,s} \|u_1\|_{\ysb}\|u_2\|_{\ysb}\|u_3\|_{\ysb}. 
\end{equation}
\end{lemma}
\begin{proof}
In the proof of \eqref{a:200}, the crux of the matter is the resonant identity 
\begin{equation}
\label{a:220}
(\tau_1+\tau_2+\tau_3)-(k_1+k_2+k_3)^3=\sum_{j=1}^3 (\tau_j-k_j^3) - 3(k_1+k_2)(k_2+k_3)(k_3+k_1).
\end{equation}
which guarantees that 
$$
\max(\tau-k^3, \tau_1-k_1^3, \tau_2-k_2^3, \tau_3-k_3^3) \gtrsim |(k_1+k_2)(k_2+k_3)(k_3+k_1)|.
$$
The corresponding ingredient needed for the proof of \eqref{a:210}, is 
\begin{eqnarray*}
& & \max(\tau-k^3-k|\hat{f}(k)|^2, \tau_j-k_j^3-k_j|\hat{f}(k_j)|^2,j=1,2,3) \gtrsim  \\
&\gtrsim& |(k_1+k_2)(k_2+k_3)(k_3+k_1)|.
\end{eqnarray*}
  This is however satisfied by an identity similar to \eqref{a:220}, since for $k_1, k_2, k_3: (k_1 + k_2)(k_2 + k_3)(k_3+ k_1) \neq 0$, 
$$
|(k_1+k_2)(k_2+k_3)(k_3+k_1)|\gtrsim k_{\max}>> O(k_{\max}^{1-2s})=|k_j|  |\hat{f}(k_j)|^2
$$
Thus, \eqref{a:210} is established. 
\end{proof} 
We now state a lemma, which allows us to place the terms like  $\sum_k \hat{f}(k) e^{i Q(t,k)} e^{i k x}$ in the space $Y^{s_0, \f{1}{2}+}$. 
\begin{lemma}
\label{le:06}
Let $b\leq 1$, $z\in H^{1-s_0}(T)$ and let $\{Q(k,t)\}_k$ be the family guaranteed to exist on $[0,T_0]$ by Lemma \ref{lo:1}. Then
$$
\|\sum_k \hat{f}(k) e^{i Q(t,k)} e^{i k x}\|_{Y_{T_0}^{s_0, b}}\leq C \sqrt{T_0}  (1+\|z\|_{H^{1-s_0}}\|f\|_{H^{s_0}(T)}) \|f\|_{H^{s_0}(T)}. 
$$
\end{lemma}
\begin{proof}
From the integral equation \eqref{k:95}, we have 
$\hat{f}(k) e^{i Q(t,k)}=\hat{f}(k) e^{it(k^3+k|\hat{f}(k)|^2)}  g(t,k)$, where  
$$
g(t,k)=exp(i(k\int_0^t (2\Re(\hat{f}(k) e^{ i Q(s,k)} \overline{\hat{z}(s,k)}) + |\hat{z}(s,k)|^2) ds)
$$
Note $|g(t,k)|=1$. 
Denote for conciseness $\phi_k=k^3+k|\hat{f}(k)|^2$, so that  \\ $
\hat{f}(k) e^{i Q(t,k)}=  e^{i t \phi_k} \hat{f}(k)   g(t,k) =: e^{i t \phi_k} \hat{h}(t,k)$. Taking Fourier transform in $t$, we have 
$$
\widehat{\hat{f}(k) e^{i Q(\cdot,k)}}(\tau)=\hat{h}(\tau-\phi_k, k).
$$
Thus, 
\begin{eqnarray*} 
\|\sum_k \hat{f}(k) e^{i Q(t,k)} e^{i k x}\|_{Y_T^{s_0,b}}^2 &=&
\sum_k <k>^{2s_0} \int <\tau-\phi_k>^{2b}|\widehat{\hat{f}(k) e^{i Q(\cdot,k)}}(\tau)|^2 d\tau =\\
&=& \sum_k <k>^{2s_0} \int <\tau-\phi_k>^{2b} |\hat{h}(\tau-\phi_k, k)|^2 d\tau=\\
&=& \|h\|_{H^b_t(0,T_0) H^s_x}^2.
\end{eqnarray*} 
We have 
\begin{equation}
\label{a:950}
\|h\|_{H^b_t H^s_x}^2 \leq \|h\|_{H^1_t(0,T_0) H^{s_0}_x}^2 \leq \sum_k <k>^{2s_0} \hat{f}(k)|^2 (\int_0^{T_0}   (1+|g'(t,k)|)^2   dt)|. 
\end{equation}
It is therefore, enough to show $\sup_k |g'(t,k)|\leq C$. But,   
\begin{eqnarray*} 
|g'(t,k)| &\leq &  |k| |\hat{z}(t,k)|( |\hat{f}(k)| + |\hat{z}(t,k)|)\leq \\
&\leq &  |k|<k>^{s_0-1}\|z\|_{H^{1-s_0}} <k>^{-s_0}\|f\|_{H^{s_0}} \leq    C\|z\|_{H^{1-s_0}} \|f\|_{H^{s_0}}, 
\end{eqnarray*} 
whence we  obtain the desired estimate. 

\end{proof}

\section{Estimates for the nonlinear terms} 
\label{sec:3}

Let $\f{1}{2}<b $ be fixed, and define  the solution space   $\cx=Y^{s_0,  b}\cap L^\infty_t H^{s_1}_x$, where $\f{1}{4}<s_0<\f{1}{2}$ and $\f{1}{2}<1-s_0<s_1< \min(1, 3s_0)$. That is 
$$
\|\cdot\|_\cx:= \|\cdot\|_{Y^{s_0,  b}}+\|\cdot\|_{L^\infty_t H^{s_1}_x}.
$$
 Note that the assumption $s_0>1/4$ is used in a crucial way to ensure that such $s_1$ exists. On the other hand $\cx\hookrightarrow L^\infty_t H^{1-s_0}$, which is used in Lemma \ref{lo:1} to justify the existence of the generalized phase function $Q_z$. 

We state several lemmas. Lemma \ref{main1}  allows us to  estimate   the contribution of all non-resonant terms, i.e. all terms appearing out of the trilinear term $\mathcal{NR}$.  The second  lemma, Lemma \ref{main2} estimates the contribution of the non-resonant terms. 
\subsection{Estimates of the non-resonant contributions}
\begin{lemma}
\label{main1}
Let   $\f{1}{4}<s_0<\f{1}{2}$. Take $\de: 0<\de<<s_0-1/4$, $b=\f{1}{2}+2 \de$ and \\ $\f{1}{2}<1-s_0<s_1<\min(1, 3s_0)$.    
For the solution to 
$$
\left|
\begin{array}{l}
\p_t \hat{U}(k)-i(k^3+k|\hat{f}(k)|^2) \hat{U}(k)=\mathcal{NR}(u_1, u_2, u_3)(k), \\
  U(0,k)=0
\end{array}
\right. 
$$
\begin{eqnarray}
\label{a:400}
\|U\|_{Y^{s_0,\f{1}{2}+\de}_T}  &\leq &  C T^\de \|u_1\|_{Y^{s_0, b}} \|u_2\|_{Y^{s_0, b}}\|u_3\|_{Y^{s_0, b}}\\
\label{a:410}
 \|U\|_{L^\infty_t(0,T) H^{s_1}_x} &\leq & C T^\de \|u_1\|_{Y^{s_0, b}} \|u_2\|_{Y^{s_0, b}}\|u_3\|_{Y^{s_0, b}}
\end{eqnarray}
  
\end{lemma}

\begin{proof}
The first estimate \eqref{a:400} is nothing but a combination\footnote{where of course the main difficulties have been  hidden behind the well-known Bourgain's  Lemma  \ref{le:bou}} of Lemma \ref{le:19} and Lemma \ref{le:bou}. We have 
\begin{eqnarray*}
& & \|U\|_{Y^{s_0,b}_T}\leq C_\de T^\de\|\mathcal{NR}(u_1, u_2, u_3)\|_{Y^{s_0,b-1+\de}_T}\leq C_\de T^\de 
\|u_1\|_{Y^{s_0,b}} \|u_2\|_{Y^{s_0,b}} \|u_3\|_{Y^{s_0,b}}. 
\end{eqnarray*}

We now take on the estimates in $L^\infty H^{s_1}$. 
We will show \eqref{a:410}  by reducing to the case when $v_1, v_2, v_3$ are free solutions in the corresponding evolutions. This is done through the well-known method of averaging (valid for general dispersion relations), which we now describe. 
Let $\mu(k)$ be a real-valued symbol, so that 
$$
X_\mu^{s,b}=\{f: T\times {\mathbf R}\to \cc: \|u\|_{X_\mu^{s,b}}^2:=\sum_k \int <\tau-\mu(k)>^{2b} |\hat{u}(\tau, k)|^2 d\tau<\infty \}
$$
Write 
\begin{equation}
\label{a:510}
u(t,x)=\int e^{i \la t} u_\la(t,x) d\la,
\end{equation}
where $\widehat{u_\la}(\tau, k)=\de(\tau-\mu(k)) \hat{u}(\tau+\la, k)$.  Clearly, $\hat{u}_\la(t,k)=e^{i t \mu(k)} \hat{u}(\la+\mu(k), k)$, that is $u_\la(t,x)$ is  a free solution of the equation 
$$
(\p_t- i\mu(-i \p_x))u_\la(t,x)=0, u_\la(0,x)=\sum_k \hat{u}(\la+\mu(k), k) e^{i k x}.
$$
Suppose that we can prove estimates for   \eqref{a:410}, where $u_j=\sum_k e^{i t \mu(k)} \hat{f}_j(k)e^{i k x}, j=1,2,3$ are free solutions,  for  $\mu(k)=k^3+k|\hat{f}(k)|^2$.  

We will provide later an almost explicit solution of \eqref{a:410}, a trilinear form \\ 
  $\cm(f_1, f_2, f_3)(t,x)=\sum_k \cm(f_1, f_2, f_3)(t,k)e^{i k x}$.  That is, we will construct
$$
\left|
\begin{array}{l}
(\p_t - i (k^3+k|\hat{f}(k)|^2)) \cm(f_1, f_2, f_3)(t,k)= \mathcal{NR}(\otimes_{j=1}^3 e^{i t \mu(k_j)}\hat{f}_j(k_j)), k=k_1+k_2+k_3\\
\cm(f_1, f_2, f_3)(0,k)=0
\end{array}
\right.
$$
 Assume for the moment the validity of 
\begin{equation}
\label{a:500}
\|\cm(f_1, f_2, f_2)\|_{L^\infty (0,T) H^{s_1}_x}\leq C \prod_{j=1}^3 \|f_j\|_{H^{s_0}}.
\end{equation}
We  show that \eqref{a:410} follows. Indeed,  employing the representation \eqref{a:510} for each of $u_j, j=1,2,3$, we have that the solution $U$ of \eqref{a:410} will take the form 
$$
U(t,x)=\int e^{i t(\la_1+\la_2+\la_3)} \cm(u_{\la_1}(0), u_{\la_2}(0), u_{\la_3}(0)) d\la_1 d\la_2 d\la_3.
$$
Taking $L^\infty_t H^{s_1}_x$ norms and applying \eqref{a:500} yields the bound 
\begin{eqnarray*}
\|U\|_{L^\infty H^{s_1}_x} &\leq & \int \| \cm(u_{\la_1}(0), u_{\la_2}(0), u_{\la_3}(0))\|_{ L^\infty H^{s_1}_x}d\la_1 d\la_2d\la_3  \leq \\
& \leq &  C  \int \|u_{\la_1}\|_{H^{s_0}}   d\la_1 \int  \|u_{\la_2}\|_{H^{s_0}}  d\la_2 \int  \|u_{\la_3}\|_{H^{s_0}} d \la_3.
\end{eqnarray*}
But 
\begin{eqnarray*}
\int \|u_{\la}\|_{H^{s_0}} d\la  &\leq & (\int <\la>^{1+2 \de}\|u_{\la}\|_{H^{s_0}}^2 d\la)^{1/2} (\int <\la>^{-1-2 \de }d\la)^{1/2} \leq  \\
& \leq & C_\de 
(\sum_k <k>^{2s_0} \int <\la>^{1+2\de}  |\hat{u}(\la+\mu(k),k)|^2d\la)^{1/2} = \\
&=&  C_\de \|u\|_{X^{s_0,\f{1}{2}+\de}_\mu}.
\end{eqnarray*}
Since $\|u\|_{X^{s_0,\f{1}{2}+\de}_{T,\mu}}\leq C_\de T^\de \|u\|_{X^{s_0,b}_\mu}$,  we have reduced matters to the construction of the trilinear form $\cm$ and  the proof of \eqref{a:500}. 
\subsubsection{Proof of \eqref{a:500}}
Introduce a notation for  the free solutions  
$$
R[g](t,x):=\sum_k e^{ i t (k^3+k|\hat{f}(k)|^2)} \hat{g}(k) e^{i k x}. 
$$
  Note  the algebraic identity 
\begin{eqnarray*}
\tau-k^3-k|\hat{f}(k)|^2 &=& \sum_{j=1}^3 (\tau_j-k_j^3-k_j|\hat{f}(k_j)|^2)-3 (k_1+k_2)(k_2+k_3)(k_3+k_1)+\\
&+& (\sum_{j=1}^3 k_j|\hat{f}(k_j)|^2) - k |\hat{f}(k)|^2
\end{eqnarray*}
for $\tau=\tau_1+\tau_2+\tau_3, k=k_1+k_2+k_3$. Denote $k_{\max}:=\max(|k_1|, |k_2|, |k_3|)$ and  
$k_{\min}:=\min(|k_1|, |k_2|, |k_3|)$, 
$$
E(k_1, k_2, k_3)= k_1 |\hat{f}(k_1)|^2+k_2 |\hat{f}(k_2)|^2+k_3 |\hat{f}(k_1)|^2- k |\hat{f}(k)|^2. 
$$
Notice that if $f\in H^{s_0}(T)$, 
\begin{eqnarray*}
& & 
|(k_1 + k_2)(k_2 + k_3)(k_3+ k_1)|\gtrsim k_{\max},  \\ 
&  & |E(k_1, k_2,. k_3)|\leq C \|f\|_{H^{s_0}(T) }^2 k_{\max}^{1-2s_0}<<k_{\max}
\end{eqnarray*} 
Thus, there exists $K_0=K_0(\|f\|_{H^{s_0}(T)}$, so that for all $k_{\max}>K_0$, we have that 
$$
|-3 (k_1+k_2)(k_2+k_3)(k_3+k_1)+E(k_1, k_2,. k_3)|\gtrsim k_{\max}>1.
$$
This allows us to define the function $h(t,x)=\sum_k \hat{h}(t,k)e^{i k x}$ 
 \begin{eqnarray*}
  \widehat{h}(t,k) &=& -\f{i}{3} \sum_{\tiny\begin{array}{c}   k=k_1+k_2+k_3\neq 0, k_{\max}>K_0 \\ (k_1 + k_2)(k_2 + k_3)(k_3+ k_1) \neq 0\end{array}} 
 \f{(k_1+k_2+k_3) \widehat{R[f_1]}(k_1) \widehat{R[f_2]}(k_2) \widehat{R[f_3]}(k_3)}{-3(k_1 + k_2)(k_2 + k_3)(k_3+ k_1)+E(k_1, k_2, k_3)}, 
 \end{eqnarray*}
since the denominator is guaranteed to stay away from zero. 

From the algebraic identity displayed above,  we see  that $h$ satisfies 
$$
(\p_t-i(k^3+k|\hat{f}(k)|^2))\hat{h}(t,k)=\mathcal{NR}^{>K_0}(R[f_1], R[f_2], R[f_3])(k),  
$$
and 
$$
\hat{h}(0,k)=  -\f{i}{3} \sum_{\tiny\begin{array}{c}   k=k_1+k_2+k_3\neq 0, k_{\max}>K_0 \\ (k_1 + k_2)(k_2 + k_3)(k_3+ k_1) \neq 0\end{array}} 
 \f{(k_1+k_2+k_3) \widehat{f_1}(k_1) \widehat{f_2}(k_2) \widehat{f_3}(k_3)}{-3(k_1 + k_2)(k_2 + k_3)(k_3+ k_1)+E(k_1, k_2, k_3)}, 
$$
where we have used the notation 
\begin{eqnarray*}
\mathcal{NR}^{\leq K_0} (v_1, v_2, v_3) & := &  -i\frac{k}{3}  
\sum_{\tiny\begin{array}{c} k_1 + k_2 + k_3 = k,\quad k_j, k\neq 0\\ (k_1 + k_2)(k_2 + k_3)(k_3+ k_1) \neq 0\\ 
|k_1|\leq K_0, |k_2|\leq K_0, |k_3|\leq K_0 \end{array}} \widehat{v_1}(k_1) \widehat{v_2}(k_2) \widehat{v_3}(k_3), \\
\mathcal{NR}^{>K_0} (v_1, v_2, v_3) &:=&  \mathcal{NR} (v_1, v_2, v_3) (k) - \mathcal{NR}^{\leq K_0} (v_1, v_2, v_3)
\end{eqnarray*}
Note that $h$ is  a trilinear form acting  on $f_1, f_2, f_3$. The construction of the $h$ provides the major step toward the construction of the $\cm$, for which we need to establish the estimate \eqref{a:500}.  In fact, we can quickly describe the remaining pieces of $\cm$. Let $h_1=\sum_k \hat{h}_1(t,k)e^{i k x}=h_1(f_1, f_2, f_3)$ satisfies 
$$
\left|
\begin{array}{l}
(\p_t-i(k^3+k|\hat{f}(k)|^2))\hat{h}_1(t,k)=\mathcal{NR}^{\leq K_0}(R[f_1], R[f_2], R[f_3])(t,k),  \\
h_1(0,k)=0
\end{array}
\right.
$$
That is 
$$
\hat{h}_1(t,k)=\int_0^t e^{i (t-s)(k^3+k|\hat{f}(k)|^2)}\mathcal{NR}^{\leq K_0}(R[f_1], R[f_2], R[f_3])(s,k)ds
$$
Finally, let $h_2=h_2(f_1, f_2, f_3)$ solves 
$$
\left|
\begin{array}{l}
(\p_t-i(k^3+k|\hat{f}(k)|^2))\hat{h}_2(t,k)=0,  \\
h_1(0,k)=-h(0,k).
\end{array}
\right.
$$
That is 
$$
h_2(t,k)=-e^{i t (k^3+k|\hat{f}(k)|^2)} \hat{h}(0,k).
$$
Clearly, 
$$
\cm(f_1, f_2, f_3)=h(f_1, f_2, f_3)+h_1(f_1, f_2, f_3)+h_2(f_1, f_2, f_3).
$$
We claim that the required estimate \eqref{a:500} follows from 
\begin{equation}
\label{a:670}
\|h(f_1, f_2, f_3)\|_{L^\infty H^{s_1}}\leq  C\prod_{j=1}^3 \|f_j\|_{H^{s_0}}.
\end{equation}
Indeed, assuming \eqref{a:670}, we have in particular 
$$
\|h(0, \cdot)\|_{H^{s_1}_x}\leq \|h(f_1, f_2, f_3)(t,\cdot)\|_{L^\infty H^{s_1}}\leq  C\prod_{j=1}^3 \|f_j\|_{H^{s_0}}.
$$
Thus, by Lemma \ref{le:19}, 
$$
\|h_2(t,\cdot)\|_{L^\infty H^{s_1}_x}\leq \|h_2(t,\cdot)\|_{Y^{s_1,b}}\leq C \|h(0, \cdot)\|_{H^{s_1}_x}\leq C\prod_{j=1}^3 \|f_j\|_{H^{s_0}}.
$$
Regarding $h_1$, we have by energy estimates 
\begin{eqnarray*}
& & \|h_1(t,\cdot)\|_{L^\infty_T 
 H^{s_1}_x}  \leq   C \|\mathcal{NR}^{\leq K_0}(R[f_1], R[f_2], R[f_3])\|_{L^1_t H^{s_1}_x}
\end{eqnarray*}
But, by H\"olders and Sobolev embedding 
\begin{eqnarray*}
& & 
\|\mathcal{NR}^{\leq K_0}(R[f_1], R[f_2], R[f_3])\|_{L^1_t H^{s_1}_x}\leq \\
&\leq & C T   (\sum_{|k|\leq 3 K_0} <k>^{2 s_1}  (\sum_{\begin{array}{c} k=k_1+k_2+k_3 \\
|k_1|\leq K_0, |k_2|\leq K_0, |k_3|\leq K_0
\end{array}}    |\hat{f}_1(k_1)| |\hat{f}_2(k_2)| |\hat{f}_3(k_3)|)^2)^{1/2} \\ 
&\leq & C K_0^{s_1} T \|(\widetilde{f_1})_{\leq K_0}  (\widetilde{f_2})_{\leq K_0}  (\widetilde{f_3})_{\leq K_0}\|_{L^2_x}
\leq C K_0^{s_1} \prod_{j=1}^3 \|(\widetilde{f_j})_{\leq K_0}\|_{L^6_x}\leq \\
& \leq & C T K_0^{s_1} \prod_{j=1}^3 \|(\widetilde{f_j})_{\leq K_0}\|_{H^{1/3}_x}\leq C T K_0^{s_1+1}  
\prod_{j=1}^3 \|\widetilde{f_j} \|_{L^2_x}\leq C T K_0^{s_1+1}  
\prod_{j=1}^3 \|f_j \|_{L^2_x}, 
\end{eqnarray*} 
where we have used the notations $\widetilde{g}(x):=\sum_k |\hat{g}(k)| e^{i k x}$ and $g_{\leq K_0}:=
\sum_{|k|<K_0} \hat{g}(k) e^{i k x}$. 

The estimates for $h_1, h_2$, in addition to \eqref{a:670} implies \eqref{a:500}. Thus, it remains to establish \eqref{a:670}. 

At this point, it is worth mentioning that the particular form of the free solutions $R[f_j]$ as entries in $h$ will not be important anymore, other than the fact that they belong to the space $H^{s_0}(T)$. Thus,  upon introducing the new trilinear form 
$$
H(v_1, v_2, v_3):=\sum_{\tiny\begin{array}{c}   k=k_1+k_2+k_3\neq 0, k_{\max}>K_0 \\ (k_1 + k_2)(k_2 + k_3)(k_3+ k_1) \neq 0\end{array}} 
 \f{(k_1+k_2+k_3) \widehat{v}_1(k_1) \widehat{v}_2(k_2) \widehat{v}_3(k_3)}{-3(k_1 + k_2)(k_2 + k_3)(k_3+ k_1)+E(k_1, k_2, k_3)}, 
$$
we will show the more general estimate 
\begin{equation}
\label{a:700}
\|H(v_1, v_2, v_3)\|_{L^\infty_t H^{s_1}_x}\leq C  \|v_1\|_{H^{s_0}(T)}\|v_2\|_{H^{s_0}(T)}\|v_3\|_{H^{s_0}(T)},
\end{equation}
which of course implies \eqref{a:670} with $v_j=R[f_j]$, since $\|v_j\|_{H^{s_0}}=\|f_j\|_{H^{s_0}}$. 

Recall $|(k_1 + k_2)(k_2 + k_3)(k_3+ k_1)|\gtrsim k_{\max}>>|E(k_1, k_2, k_3)|$. Thus, we have the following  inequalities 
\begin{eqnarray*}
& & \f{|k_1+k_2+k_3|}{|-3(k_1 + k_2)(k_2 + k_3)(k_3+ k_1)+E(k_1, k_2, k_3)|}\leq \\
& \leq & C \left|  \f{k_1+k_2+k_3}{(k_1 + k_2)(k_2 + k_3)(k_3+ k_1)}\right|\leq \\
& \leq & \f{C}{|k_1 + k_2| |k_2 + k_3|}+\f{C}{|k_2 + k_3| |k_3 + k_1|}+\f{C}{|k_1 + k_2| |k_3 + k_1|}. 
\end{eqnarray*}
 We need to consider two cases - $k_{\min}\sim k_{\max}$ and the  case $k_{\min}<<k_{\max}$.  \\
\\
{\bf Case I: $k_{\min}\sim k_{\max}$ or $|k_1|\sim |k_2|\sim |k_3|$.}  In this case $|k|\lesssim |k_j|, j=1,2,3$. We only consider the term $\f{1}{(k_1 + k_2)(k_2 + k_3)}=\f{1}{(k_1 + k_2)(k- k_1)}$, the others being symmetric. 

By Cauchy-Schwartz, we have  
\begin{eqnarray*}
& & |H(v_1, v_2, v_3)(k)|^2\leq   (\sum_{k_1, k_2: |k_1|\sim |k_2|\gtrsim |k| }  |\hat{v}_1(k_1)|^2 |\hat{v}_2(k_2)|^2)  \times \\ 
&\times &  (\sum_{k_1, k_2: |k_1|\sim |k_2|\sim |k-k_1-k_2|\gtrsim |k|}   \f{\hat{v}_3(k-k_1-k_2)|^2}{|k_1+k_2|^2 |k-k_1|^2 } ) 
\end{eqnarray*}
It is now easy to estimate  
\begin{eqnarray*}
& & \|H(v_1, v_2, v_3)\|_{H^{s_1}}^2\leq C \sum_k <k>^{2s_1} 
(\sum_{k_1: |k_1|\gtrsim |k| }  |\hat{v}_1(k_1)|^2) (\sum_{k_2: |k_2|\gtrsim |k| }  |\hat{v}_2(k_2)|^2) \times \\
&\times &   (\sum_{k_1, k_2: |k-k_1-k_2|\gtrsim |k|} |\hat{v}_3(k-k_1-k_2)|^2\f{1}{|k_1+k_2|^2 |k-k_1|^2})\leq  \\
&\leq & C (\sum_{k_1  } <k_1>^{2s_1/3}  |\hat{v}_1(k_1)|^2) (\sum_{k_2  } <k_2>^{2s_1/3}   |\hat{v}_2(k_2)|^2)\times \\
&\times & 
\sum_{ \mu, k_1, k_2:
(k_1+k_2)(\mu+k_2)\neq 0} <\mu>^{2s_1/3} |\hat{v}_3(\mu)|^2 \f{1}{|k_1+k_2|^2 |\mu+k_2|^2} 
\leq  \\
&\leq & C  \|v_1\|_{H^{s_0}}^2  \|v_2\|_{H^{s_0}}^2 \|v_3\|_{H^{s_0}}^2. 
\end{eqnarray*} 
provided  $s_1<3s_0$, since   
$$
\sum_{k_1, k_2: (k_1+k_2)(\mu+k_2)\neq 0} \f{1}{|k_1+k_2|^2 |\mu+k_2|^2}<\infty.
$$
{\bf Case II: $k_{\min}<<k_{\max}$.} 
In this case, we have that for all $i\neq j\neq l\neq i$, \\ 
$|(k_i+k_j)(k_j+k_l)|\gtrsim  k_{\max}$. Thus 
\begin{eqnarray*}
& & |H(v_1, v_2, v_3)(k)|\leq    C <k>^{-1} \sum_{k_1, k_2} |\hat{v}_1(k_1)| |\hat{v}_2(k_2)| |\hat{v}_3(k-k_1-k_2)| 
\end{eqnarray*}
Since  $1>s_1>1/2$. We have by Sobolev embedding\footnote{recall that we use the notation $\tilde{v}(x)=\sum_k |\hat{v}(k)| e^{i k x}$} 
\begin{eqnarray*}
& & \|H(v_1, v_2, v_3)\|_{H^{s_1}}\leq C \||\p_x|^{s_1-1}[ \tilde{v}_1 \tilde{v}_2 \tilde{v}_3]\|_{L^2}  \leq   C
\|  \tilde{v}_1 \tilde{v}_2 \tilde{v}_3\|_{L^q} \leq  \\
&\leq & C  \|  \tilde{v}_1 \|_{L^{3 q}} \|  \tilde{v}_2 \|_{L^{3 q}} \|  \tilde{v}_3 \|_{L^{3 q}}
\end{eqnarray*}
 where $\f{1}{q}-\f{1}{2}=1-s_1$, so that $q\in (1,2)$.  Under the restriction    $s_1<\min(3s_0,1)$, it follows  by Sobolev embedding 
$$
\|\tilde{v}_j\|_{L^{3q}}\leq C \|\tilde{v}_j\|_{H^{s_1/3}}\leq C  \|v_j\|_{H^{s_0}}.
$$
since $\f{1}{2}-\f{1}{3q}=\f{s_1}{3}<s_0$.  This finishes the proof of the estimate \eqref{a:700} and hence the proof of Lemma \ref{main1}. 
\end{proof}

 \subsection{Estimate of the resonant contributions}
\begin{lemma}
\label{main2}
Let $\f{1}{2}>s_0>\f{1}{4}$, $\de: \de<<s_0-1/4$, $b=1/2+\de$, $1-s_0<s_1<\min(1, 3s_0)$.  
Assume that $F_1, F_2; G_1, G_2\in L^\infty_T H^{s_1}(T)$, whereas 
$v_1, v_2\in Y^{s_0,b}$. For the solution of 
$$
\left|
\begin{array}{l}
\p_t \hat{V}(k)- i(k^3+k|\hat{f}(k)|^2) \hat{V}(k)= c_1 k \hat{v}_1(k) \hat{F}_1 (k)\overline{\hat{F}_2}(k) + 
c_2 k \overline{\hat{v}_2} (k) \hat{G}_1 (k) \hat{G}_2(k) \\
\hat{V}(0,k)=0
\end{array}
\right.
$$
we have the estimates, with $C=C(c_1, c_2)$ 
\begin{eqnarray}
\label{b:10}
& & \|V\|_{Y^{s_0,b}_T}\leq C (\|v_1\|_{Y^{s_0,b}} \|F_1\|_{L^\infty H^{s_1}}\|F_2\|_{L^\infty H^{s_1}}  + 
\|v_2\|_{Y^{s_0,b}} \|G_1\|_{L^\infty H^{s_1}}\|G_2\|_{L^\infty H^{s_1}})\\
\label{b:20}
& &  \|V\|_{L^\infty H^{s_1}}\leq  C (\|v_1\|_{Y^{s_0,b}} \|F_1\|_{L^\infty H^{s_1}}\|F_2\|_{L^\infty H^{s_1}}  + 
\|v_2\|_{Y^{s_0,b}} \|G_1\|_{L^\infty H^{s_1}}\|G_2\|_{L^\infty H^{s_1}})
\end{eqnarray}
 
\end{lemma}

\subsection{Proof of Lemma \ref{main2}} 
The proof of Lemma \ref{main2} is fairly easy. Denote the right hand side of the   equation by $RHS$.  By Lemma \ref{le:19}, 
\begin{eqnarray*}
& & \|V\|_{Y^{s_0,b}_T}\leq C_\de T^\de \|RHS\|_{Y^{s_0, b-1+\de}}\leq C T^\de \|RHS\|_{L^2_T H^{s_0}_x}\leq C T^{\de+1/2} \sup_t \|RHS\|_{H^{s_0}}.
\end{eqnarray*}
By energy estimates 
$$
\|V\|_{L^\infty_T H^{s_1}}\leq C \|RHS\|_{L^1_t H^{s_1}_x}\leq C T \sup_t \|RHS\|_{H^{s_1}_x}.
$$
Thus, recalling that $T<1$, $\|V\|_{Y^{s_0,b}_T}+ \|V\|_{L^\infty_T H^{s_1}}\leq C \sqrt{T}  \sup_t \|RHS\|_{H^{s_1}_x}$, so it suffices to estimate this quantity.  We also estimate only say the first quantity of $RHS$, since they are symmetric from the point of view of the required estimates. We have 
\begin{eqnarray*}
  \|RHS\|_{H^{s_1}_x}^2 &\leq & C \sum_k <k>^{2(1+s_1)}  |\hat{v}_1(k)|^2  |\hat{F}_1 (k)|^2 \hat{F}_2(k)|^2 \leq  \\
&\leq &  C  (\sup_k <k>^{s_0} |\hat{v}_1(k)|)^2 (\sup_k <k>^{1-s_0} |\hat{F}_2(k)|)^2 \sum_k <k>^{2s_1} |\hat{F}_1 (k)|^2 \\ &\leq & C \|v_1\|_{H^{s_0}}^2 \|F_1\|_{H^{1-s_0}}^2  \|F_1\|_{H^{s_1}}^2 
\end{eqnarray*}
The estimate follows since $1-s_0<s_1$.

\section{Proof of Theorem \ref{theo:1}} 
\label{sec:4}

\subsection{Existence of the solution}
We start with the existence of the solution $z$ in the sense of Definition  \ref{defi:1}. We produce it by an iteration argument as follows\footnote{recall that $Q=Q(z)$ is constructed for {\it a given $z$} in Lemma \ref{lo:1}}. 
Start with $z_0=0$ and $Q_0(t)=t(k^3+k |\hat{f}(k)|^2)$ as prescribed in Lemma \ref{lo:1}. Define iteratively, $z_{m+1}, m=0, \ldots$ by producing the next iterate from the previous one, namely 
 \begin{eqnarray*}
  \hat{z}_{m+1}(t,k) &=& \int_0^t e^{i (t-s) (k^3+k|\hat{f}(k)|^2)} [i k |\hat{z}_m(s,k)|^2 \hat{z}_m(s,k)+   \\
&+& \int_0^t e^{i (t-s) (k^3+k|\hat{f}(k)|^2)} 2i k \Re(\hat{f}(k) e^{i Q_m(s,k)}\overline{\hat{z}_m(t,k)}) \hat{z}_m(k)] ds + \\
& + & \int_0^t e^{i (t-s) (k^3+k|\hat{f}(k)|^2)} [\mathcal{NR}(\otimes_{j=1}^3  \hat{f}(k_j) e^{i  Q_m(t,k_j)} +\hat{z}_m(k_j))]ds.
\end{eqnarray*}
By the definition, 
$$
\hat{z}_1(t,k)= \int_0^t e^{i (t-s) (k^3+k|\hat{f}(k)|^2)} [\mathcal{NR}(\otimes_{j=1}^3  \hat{f}(k_j) e^{i  t(k^3+k|\hat{f}(k)|^2} ]ds.
$$
According to the estimates in Lemma \ref{main1}, we have that 
$$
\|z_1\|_{\cx}\leq C  \|f\|_{H^{s_0}}^3.
$$
Denote $K:=\|z_1\|_{\cx}< C  \|f\|_{H^{s_0}}^3$. We will show that with the right choice of $T$ (to be made precise below), we will have that $\|z_j\|_{\cx}\leq 2 K$. 

 We need to estimate $\|z_{m+1}-z_m\|_\cx$. The right hand side of the equation for $z_{m+1}$ has a multilinear structure, which allows us (by adding and subtracting  appropriate terms) to use the estimates of Lemma \ref{main1} and Lemma \ref{main2}.  Denote for conciseness  \\ $F_m(t,x):=\sum_k \hat{f}(k) e^{i Q_m(t,k)}e^{i k x}$. 
 We have 
\begin{eqnarray*}
& & \|z_{m+1}-z_m\|_\cx\lesssim T^\de \|z_{m}-z_{m-1}\|_\cx (\|z_m\|_\cx+\|z_{m-1}\|_\cx)^2+ \\
&+& T^\de\|z_{m}-z_{m-1}\|_\cx (\|z_m\|_\cx+\|z_{m-1}\|_\cx+\|F_m \|_{Y^{s_0,b}}+\|F_{m-1} \|_{Y^{s_0,b}})^2\\
&+& T^\de  \|F_m -F_{m-1} \|_{Y^{s_0,b}}(\|z_m\|_\cx+\|z_{m-1}\|_\cx+\|F_m \|_{Y^{s_0,b}}+\|F_{m-1} \|_{Y^{s_0,b}})^2.
\end{eqnarray*}
Further,  similar to Lemma \ref{le:06} (more specifically \eqref{a:950}),  we estimate, 
\begin{equation}
\label{c:15}
 \|F_m -F_{m-1} \|_{Y^{s_0,b}_T}\lesssim \left(  \sum_k <k>^{2s_0} |\hat{f}(k)|^2 \int_0^{T} (1+   |g'(t,k)|^2) dt
      \right)^{1/2}
\end{equation}
    where $g(t,k)=e^{i Q_m(t,k)}-e^{i Q_{m-1}(t,k)}$. But 
$$
|g'(t,k)|\leq C[ |Q_{m-1}'(t,k)| |Q_m(t,k)-Q_{m-1}(t,k)|+|Q'_m(t,k)-Q'_{m-1}(t,k)|].
$$
From \eqref{k:100}, we have 
\begin{eqnarray*} 
& & |Q'_m(t,k)-Q'_{m-1}(t,k)| \leq   C \sup_{0<t<T_0} |k| |\hat{f}(k)| \times \\
&\times& ( |Q_m(t,k)-Q_{m-1}(t,k)|+|\hat{z}_m(t,k)-\hat{z}_{m-1}(t,k)|)
(|\hat{z}_m(t,k)|+|\hat{z}_{m-1}(t,k)|)  
\end{eqnarray*} 
Employing the estimates of Lemma \ref{le:06}, namely the bound  
$$
\sup_{0<t<T_0} |k| |\hat{f}(k)| |\hat{z}(t,k)|\leq C \|f\|_{H^{s_0}}\|z\|_{\cx}, 
$$
we conclude 
\begin{eqnarray*}
|Q'_m(t,k)-Q'_{m-1}(t,k)| &\leq &  C\|f\|_{H^{s_0}}(\|z_m\|_{\cx}+\|z_{m-1}\|_{\cx}) |Q_m(t,k)-Q_{m-1}(t,k)|+ \\
&+& C \|f\|_{H^{s_0}}(\|z_m\|_{\cx}+\|z_{m-1}\|_{\cx}) \|z_{m}-z_{m-1}\|_{\cx}. 
\end{eqnarray*} 
Similarly, 
\begin{eqnarray*} 
|Q_{m-1}'(t,k)| &\leq  & C\sup_{0<t<T} |k|   |\hat{z}_{m-1}(t,k)|(|\hat{f}(k)|+ |\hat{z}_{m-1}(t,k)|)\leq \\
&\leq & C \|z_{m-1}\|_{\cx}(\|f\|_{H^{s_0}}+\|z_{m-1}\|_{\cx}).
\end{eqnarray*}  
Putting all estimates together yields 
\begin{eqnarray*}
|g'(t,k)| &\leq & C \|f\|_{H^{s_0}}(\|z_m\|_{\cx}+\|z_{m-1}\|_{\cx})|Q_m(t,k)-Q_{m-1}(t,k)|+ \\
&+& \|f\|_{H^{s_0}}(\|z_m\|_{\cx}+\|z_{m-1}\|_{\cx}) \|z_{m}-z_{m-1}\|_{\cx}.
\end{eqnarray*} 
Thus, we now need to find a good estimate for $|Q_m(t,k)-Q_{m-1}(t,k)|$. Arguing again from the integral equation \eqref{k:95}, we have 
\begin{eqnarray*} 
& & 
|Q_m(t,k)-Q_{m-1}(t,k)| \leq  \\
&\leq & C T |k||\hat{f}(k)| \sup_{0\leq \tau<t}   |Q_m(\tau,k)-Q_{m-1}(\tau,k)| 
(|\hat{z}_m(t,k)|+|\hat{z}_{m-1}(t,k)|)+\\
&+& C T   |k| \sup_{0\leq \tau<t}  (|\hat{z}_m(t,k)|+|\hat{z}_{m-1}(t,k)|) |\hat{z}_m(t,k)-\hat{z}_{m-1}(t,k)|\leq \\
&\leq & C T \|f\|_{H^{s_0}} (\|z_m\|_{\cx}+\|z_{m-1}\|_{\cx}) \sup_{0\leq \tau<t}   |Q_m(\tau,k)-Q_{m-1}(\tau,k)|  +\\
&+&  C \|z_m-z_{m-1}\|_{\cx} (\|z_m\|_{\cx}+\|z_{m-1}\|_{\cx}). 
\end{eqnarray*}
Now, if $T$ is so small that $C T \|f\|_{H^{s_0}} (\|z_m\|_{\cx}+\|z_{m-1}\|_{\cx})\leq \f{1}{2}$, we can hide the first term on the right hand side and thus, we obtain the estimate 
$$
\sup_{0<t<T} |Q_m(t,k)-Q_{m-1}(t,k)| \leq C \|z_m-z_{m-1}\|_{\cx} (\|z_m\|_{\cx}+\|z_{m-1}\|_{\cx}). 
$$
In all 
$$
|g'(t,k)|\leq C  \|z_m-z_{m-1}\|_{\cx} \|f\|_{H^{s_0}}(\|z_m\|_{\cx}+\|z_{m-1}\|_{\cx})^2.
$$
Hence, plugging this back in \eqref{c:15},  we obtain 
\begin{equation}
\label{a:960}
\|F_m -F_{m-1} \|_{Y^{s_0,b}_T}\leq C    \|z_m-z_{m-1}\|_{\cx} \|f\|_{H^{s_0}}^2 (\|z_m\|_{\cx}+\|z_{m-1}\|_{\cx})^2,
\end{equation}
under the additional smallness assumption on $T: T  \|f\|_{H^{s_0}} (\|z_m\|_{\cx}+\|z_{m-1}\|_{\cx})<<1$. 

Going further back to our estimate for $\|z_{m+1}-z_m\|_{\cx}$, and plugging in \eqref{a:960}, we have 
\begin{eqnarray*}
\|z_{m+1}-z_m\|_{\cx}\leq C T^\de   \|z_m-z_{m-1}\|_{\cx} (1+\|z_m\|_\cx+\|z_{m-1}\|_\cx+\|f \|_{H^{s_0}})^4
\end{eqnarray*}
 where we have also used Lemma \ref{le:06}, to control $\|F_m\|_{Y^{s_0,b}},\|F_{m-1}\|_{Y^{s_0,b}} \leq C_T \|f\|_{H^{s_0}}.$ 

Clearly, one can choose now $T$, so that $T$ satisfies the previous assumptions \\  (i.e. $T \|f\|_{H^{s_0}} (\|z_m\|_{\cx}+\|z_{m-1}\|_{\cx})<<1$) and \\ 
$T: T^\de(1+\|z_m\|_\cx+\|z_{m-1}\|_\cx+\|f \|_{H^{s_0}})^4<\f{1}{2}$. This will ensure that 
$$
\|z_{m+1}-z_m\|_{\cx}\leq \f{1}{2} \|z_{m}-z_{m-1}\|_{\cx},
$$
and thus Cauchyness and the convergence of $\{z_m\}$, $z:=\lim_m z_m$, where $z:[0,T]\to \cc$.  In addition, 
$$
\|z\|_{\cx}\leq \|z_1\|_{\cx}+ \sum_{m=2}^\infty \|z_m-z_{m-1}\|_{\cx}\leq 2K,
$$
where $\|z_1\|_{\cx}=K$. This completes the existence part of the argument. 

\subsection{Smoothing effects}
The first smoothing effect announced in  Theorem \ref{theo:1} follows from  $z\in \cx\hookrightarrow L_t^\infty H_x^{s_1}\subset L^\infty H^{3s_0-}$. 

For \eqref{a:1050}, 
we have 
$$
|\hat{u}(t,k)|^2 = |\hat{z}(t,k)|^2+ |\hat{f}(k)|^2 + 2\Re (\hat{f}(k) e^{i Q(t,k)}\overline{\hat{z}(t,k)})
$$
whence, since $s_1+s_0>1$
\begin{eqnarray*}
& &\sup_t  \sum_k |k| ||\hat{u}(t,k)|^2-|\hat{f}(k)|^2|\leq 2 \sum_k |k|   |\hat{z}(t,k)|(|\hat{z}(t,k)|+ |\hat{f}(k)|)\leq \\
&\leq & C (\sum_k <k>^{2s_0} ( |\hat{f}(k)|^2+ |\hat{z}(t,k)|^2))^{1/2} (\sum_k <k>^{2s_1} |\hat{z}(t,k)|^2)^{1/2} \leq \\
& \leq & C \|z\|_{H^{s_1}}(\|f\|_{H^{s_0}}+ \|z\|_{H^{s_0}})\leq C\|z\|_{\cx}(\|f\|_{H^{s_0}}+\|z\|_{\cx}). 
\end{eqnarray*}

\subsection{Uniqueness} 
The uniqueness of the solution, in the sense of Definition \ref{defi:1} requires us to analyze \eqref{v} in detail. We start with the proof of the well-posedness of \eqref{v}. 
\subsubsection{Proof of the well-posedness of \eqref{v}, for fixed $u$} 
Let us first show that under the condition \eqref{cond:10}, the equation  \eqref{v} produces unique local solutions in $H^{s_1}$, recall  $s_1<\min(3s_0,1)$. The main ingredient  that we need  here is  
$$
\sum_k \hat{f}(k) e^{i P(t,k)} e^{i k x} \in Y^{s_0,b},
$$
which is simply a variant of Lemma \ref{le:06}. Indeed, observe that 
$$
P(u;t,k)= i t (k^3+k|\hat{f}(k)|^2)+k \int_0^t (|\hat{u}(s,k)|^2-|\hat{f}(k)|^2) ds. 
$$
Thus, similar to the proof of Lemma \ref{le:06}, we infer  the bound 
\begin{equation}
\label{a:1010}
\|\sum_k \hat{f}(k) e^{i P(t,k)} e^{i k x}\|_{Y^{s_0,b}_T}\leq C\sqrt{T}\|f\|_{H^{s_0}(\mathbf{T})},
\end{equation}
provided we can show  $\sup_{k,t} |(e^{i k \int_0^t (|\hat{u}(s,k)|^2-|\hat{f}(k)|^2) ds})'|<C.$ But  by \eqref{cond:10} 
$$
\sup_{k,t} |(e^{i k \int_0^t (|\hat{u}(s,k)|^2-|\hat{f}(k)|^2) ds})'|=\sup_{k,t}  |k| \left||\hat{u}(t,k)|^2-|\hat{f}(k)|^2\right|<C,
$$
 and hence 
the solutions in \eqref{v} are in $H^{s_1}$, in some time interval $[0,T], T=T(\|f\|_{H^{s_0}})$. In addition, there is the estimate 
$$
\|v\|_{\cx_T}\leq C_T  \|f\|_{H^{s_0}}^3
$$
There is an unique solution $v$ in this class. Indeed, we have the multilinear structure of the non-linearity, which allows us to use Lemma \ref{main1} and Lemma \ref{main2} to show that it is a contraction on the space $\cx_T$, whence uniqueness follows. This, however does not, by itself  imply uniqueness due to its dependence on $P=P(u)$.  Let us explain this point in more detail.  So far, we have shown that for a given $u$, with the property \eqref{cond:10}, the equation \eqref{v} has an unique  solution $v$. For the uniqueness, we need to establish more. Namely that for two different $u_1, u_2$ and the corresponding $v_1, v_2$, constructed via \eqref{v}, where $P(u_j,t,k)$ are involved, we still have $v_1=v_2$ (which then will later easily imply $u_1=u_2$). 
 \subsubsection{Estimate on the difference $v_1-v_2$}
    Taking the difference of $v_1, v_2$, we see that it  satisfies  an equation similar to the one satisfied by $z_{m+1}-z_m$ that we have considered for the existence part. 
Using the multilinear structure and the estimates of Lemma \ref{main1}, Lemma \ref{main2}, we obtain 
\begin{eqnarray*}
& & \|v_1-v_2\|_{\cx} \lesssim T^\de \|v_1-v_2\|_\cx (\|v_1\|_\cx+\|v_2\|_\cx+\|F_1 \|_{Y^{s_0,b}}+\|F_{2} \|_{Y^{s_0,b}})^2+ \\
&+& T^\de  \|F_1 -F_{2} \|_{Y^{s_0,b}} (\|v_1\|_\cx+\|v_2\|_\cx+\|F_1 \|_{Y^{s_0,b}}+\|F_{2} \|_{Y^{s_0,b}})^2. 
\end{eqnarray*}
where again, we have adopted the notation $P_j(t,k)=P(u_j;t,k)$ and \\ $F_j:=\sum_k \hat{f}(k) e^{i P_j(t,k)} e^{i k x}$. In view of our bound \eqref{a:1010}, we have 
\begin{equation}
\label{a:1030}
\|v_1-v_2\|_{\cx} \lesssim T^\de (\|v_1-v_2\|_\cx+ \|F_1 -F_{2} \|_{Y^{s_0,b}} ) (1+\|f\|_{H^{s_0}}^3)^2.
\end{equation}
Thus, our main task  now is to effectively control $ \|F_1-F_2\|_{Y^{s_0,b}}$. To that end, represent 
\begin{eqnarray*} 
F_1-F_2 &=& \sum_k \hat{f}(k) (e^{i P_1(t,k)}-e^{i P_2(t,k)}) e^{i k x} = \\
&=& \sum_k \hat{f}(k) e^{i t (k^3+k|\hat{f}(k)|^2)} e^{i k x} 
(e^{i k \int_0^t ( |\hat{u_1}(s,k)|^2-|\hat{f}(k)|^2)ds }-  e^{i k \int_0^t (|\hat{u_2}(s,k)|^2-|\hat{f}(k)|^2) ds }).
\end{eqnarray*} 
Similar to \eqref{c:15}, we can estimate 
$$
\|F_1-F_2\|_{Y^{s_0,b}}\leq C \|f\|_{H^{s_0}} \sup_k |g'(t,k)|,
$$
where  
$$
g(t,k)=e^{i k\int_0^t ( |\hat{u_1}(s,k)|^2-|\hat{f}(k)|^2)ds }-  e^{i k \int_0^t (|\hat{u_2}(s,k)|^2-|\hat{f}(k)|^2) ds }. 
$$
Adding and subtracting terms and using the a-priori bound \eqref{cond:10}  yields  
\begin{eqnarray*} 
|g'(t,k)| &\leq& C | k| \left| |\hat{u}_1(t,k)|^2-|\hat{f}(k)|^2\right| |k| T \sup_{0<\tau<T}  \left| |\hat{u}_1(\tau,k)|^2-
|\hat{u}_2(\tau,k)|^2\right|+ \\
&+& |k| | |\hat{u}_1(t,k)|^2-|\hat{u}_2(t,k)|^2|\leq \tilde{C} (1+T) |k|  \sup_{0<\tau\leq T}  \left| |\hat{u}_1(\tau,k)|^2-|\hat{u}_2(\tau,k)|^2\right|.
\end{eqnarray*} 
Thus, we need   control in the form (for say  $T\leq 1$) 
\begin{equation}
\label{a:1020}
\sup_{k } |k|    \left| |\hat{u}_1(t,k)|^2-|\hat{u}_2(t,k)|^2\right|\leq C(\|f\|_{H^{s_0}}) \|v_1-v_2\|_{\cx_T}.
\end{equation}
Let us show that once we assume  \eqref{a:1020}, we can establish the uniqueness. Indeed, plugging \eqref{a:1020} in the estimate for $|g'(t,k)|$, we obtain 
$$
\|F_1-F_2\|_{Y^{s_0,b}}\leq  C(\|f\|_{H^{s_0}}) \|v_1-v_2\|_{\cx_T}
$$ 
Going back to \eqref{a:1030}, we have (for say all $T: T<1$) 
$$
\|v_1-v_2\|_{\cx} \leq C(\|f\|_{H^{s_0}}) T^\de   \|v_1-v_2\|_{\cx_T}  (1+\|f\|_{H^{s_0}}^3)^2,
$$
which imply that for small enough $T=T(\|f\|_{H^{s_0}})$, $\|v_1-v_2\|_{\cx_T}=0$. 

Thus, again from \eqref{a:1020}, we obtain that $|u_1(t,k)|=|u_2(t,k)|$, which implies that $P_1(t,k)=P_2(t,k)$. This however means that $u_1=u_2$, so uniqueness follows. 
\subsubsection{Proof of \eqref{a:1020}} 
Expanding $|\hat{u}_j(t,k)|^2$ and taking the difference yields 
\begin{eqnarray*}
|\hat{u}_1(t,k)|^2-|\hat{u}_2(t,k)|^2 &=&  2 \Re(\hat{f}(k) (e^{i P_1(t,k)} \overline{\hat{v}_1(t,k)}-e^{i P_2(t,k)} \overline{\hat{v}_2(t,k)})) \\
&+& |\hat{v}_1(t,k)|^2 - |\hat{v}_2(t,k)|^2.
\end{eqnarray*} 
Thus, 
\begin{eqnarray*}
\left| |\hat{u}_1(t,k)|^2-|\hat{u}_2(t,k)|^2\right| &\leq &  C |\hat{f}(k)| (| \hat{v}_1(t,k)- \hat{v}_2(t,k)|+ |\hat{v}_1(t,k)|
|P_1(t,k)-P_2(t,k)|) + \\
&+& |\hat{v}_1(t,k)-\hat{v}_2(t,k)| (|\hat{v}_1(t,k)|+|\hat{v}_2(t,k)|). 
\end{eqnarray*} 
But 
$$
|P_1(t,k)-P_2(t,k)|\leq C T |k| \sup_{0<\tau<t} \left| |\hat{u}_1(t,k)|^2-|\hat{u}_2(t,k)|^2\right|.
$$
Thus, we have 
\begin{eqnarray*}
\left| |\hat{u}_1(t,k)|^2-|\hat{u}_2(t,k)|^2\right| &\leq & C T  
 \sup_{0<\tau<t} \left| |\hat{u}_1(\tau,k)|^2-|\hat{u}_2(\tau,k)|^2\right| |k||\hat{f}(k)| |\hat{v}_1(t,k)|+ \\
&+& C |\hat{v}_1(t,k)-\hat{v}_2(t,k)| (|\hat{v}_1(t,k)|+|\hat{v}_2(t,k)|+|\hat{f}(k)|)
\end{eqnarray*} 
We can now run a continuity argument in $A(t,k):=\sup_{0<\tau<t} \left| |\hat{u}_1(\tau,k)|^2-|\hat{u}_2(\tau,k)|^2\right|$, since (recalling that $s_0+s_1>1$)
$$
\sup_k |k||\hat{f}(k)| |\hat{v}_1(t,k)|\leq C \|f\|_{H^{s_0}} \|v_1\|_{H^{s_1}}\leq C \|f\|_{H^{s_0}} \|v_1\|_{\cx}.
$$
We have 
$$
A(t)\leq [C T \|f\|_{H^{s_0}} \|v_1\|_{\cx} ]A(t)+ C |\hat{v}_1(t,k)-\hat{v}_2(t,k)| (|\hat{v}_1(t,k)|+|\hat{v}_2(t,k)|+|\hat{f}(k)|).
$$
Thus, for $T$ small enough, $T=T(\|f\|_{H^{s_0}})$ (recall the bounds on $\|v_1\|_{\cx}$ are in terms of $C \|f\|_{H^{s_0}}^3$), we can hide the terms containing $A(t)$ on the right hand side.  We obtain 
$$
\left|\hat{u}_1(t,k)|^2-|\hat{u}_2(t,k)|^2\right|\leq A(t) \leq C |\hat{v}_1(t,k)-\hat{v}_2(t,k)| (|\hat{v}_1(t,k)|+|\hat{v}_2(t,k)|+|\hat{f}(k)|).
$$
It follows that (again, since $s_0+s_1>1$) 
\begin{eqnarray*}
|k| \left| |\hat{u}_1(t,k)|^2-|\hat{u}_2(t,k)|^2\right| &\leq &  C |k| |\hat{v}_1(t,k)-\hat{v}_2(t,k)| 
(|\hat{v}_1(t,k)|+|\hat{v}_2(t,k)|+|\hat{f}(k)|)\leq \\
&\leq & C \|v_1-v_2\|_{H^{s_1}}
(\|v_1\|_{H^{s_0}}+\|v_2\|_{H^{s_0}}+
\|f\|_{H^{s_0}})\leq \\
&\leq & C \|v_1-v_2\|_{\cx} (1+ \|f\|_{H^{s_0}}^3),
\end{eqnarray*}
which is \eqref{a:1020}. Thus, the uniqueness and thus the proof of Theorem \ref{theo:1} is complete.

\end{document}